\documentclass{conm-p-l}
\usepackage{amssymb}
\newtheorem{theorem}{Theorem}
\newtheorem*{theorem*}{Theorem}
\newtheorem{lemma}[theorem]{Lemma}
\newtheorem{proposition}[theorem]{Proposition}
\newtheorem{corollary}[theorem]{Corollary}

\theoremstyle{remark}
\newtheorem*{remark*}{Remark}

\begin{document}
\title[Existence of GCD's and Factorization]
{Existence of GCD's and Factorization in Rings
of non-Archimedean Entire
Functions}
\author{William Cherry}
\thanks{Partial financial support provided
by the National Security Agency under
Grant Number H98230-07-1-0037.
The United States Government is authorized to reproduce and
distribute reprints not­withstanding any copyright notation herein.}
\address{Department of Mathematics\\University of North Texas\\ 
1155 Union Circle \#311430, Denton, TX  76203\\USA}
\email{wcherry@unt.edu}
\date{March 16, 2011}
\dedicatory{Dedicated to the memory of Nicole De Grande-De Kimpe
and to C.-C.~Yang and Alain Escassut commemorating the occasions
of their 65-th birthdays.}
\begin{abstract} A detailed proof is given of the well-known facts
that greatest common divisors exist in rings of non-Archimedean
entire functions of several variables and that these rings of entire
functions are almost factorial, in the sense that an entire function
can be uniquely written as a countable product of irreducible entire
functions.
\end{abstract}
\maketitle
In \cite{CY}, Ye and I  needed the fact that greatest common divisors
exist in rings of non-Archimedean entire functions of several variables.
In that paper, we wrote:
\begin{quote}
``by standard arguments (see any book
on several complex variables that discusses the Second Cousin Problem
and the Poincar\'e Problem), we need only consider \dots.''
\end{quote}
We then gave an argument of L\"utkebohmert \cite{Lu} that the essential
property held.  We left it to the reader to fill in the details that this
really did imply the existence of gcd's.  Cristina Toropu, now a 
Ph.D.\ student at the University of New Mexico, asked me to write up
a detailed discussion of the details Ye and I omitted from
the appendix to \cite{CY}.
The result is this short note, intended primarily for students and others
new to the subject of non-Archimedean analysis.  The arguments presented
are standard, but not, as far as I know, available in the literature in the
context of non-Archimedean entire functions. Part of what
I present closely parallels section~6.4 in Krantz \cite{Krantz}, where
he discusses algebraic properties of rings of analytic functions in
several complex variables.  Thus, this note also serves to illustrate
a useful principle that someone new to the subject should keep in mind:
\textit{if a theorem in complex analysis makes use of the local ring of germs
of analytic functions at a point, the appropriate substitute in 
non-Archimedean analysis is the ring of analytic functions
on a closed ball.}

\smallskip
The purpose of this note is to illustrate how
one transfers a local algebraic property, in this case the existence of
greatest common divisors in the ring of analytic functions on a closed
ball, to the global ring of entire functions. The algebraic properties of 
rings of analytic functions on closed balls, or more generally affinoid
domains, is broadly treated in books, and so I refer, for instance,
to \cite{BGR} for the fact that the ring of analytic functions on a
closed ball is factorial and for the proof of the Weierstrass Preparation
theorem.

\smallskip
I would like to emphasize that this note concerns functions of
\textit{several} variables. In one variable, it is not hard to see that
a non-Archimedean entire function factors into an infinite product of
the form
$$
         cz^e\prod_{i\in I}\left(1-\frac{z}{a_i}\right)^{e_i},
$$
where $c\in\mathbf{F},$ $e$ is a non-negative integer,
$I$ is a countable index set, the $e_i$ are positive integers,
and the $a_i$ are non-zero elements of
$\mathbf{F}$ with at most finitely many $a_i$
in any bounded subset of $\mathbf{F};$ compare with
Theorem~\ref{almostufd}. See \cite{Lazard} for a detailed treatment
of the one variable case.

\medskip
Given that this note resulted from discussions with a student and is
intended primarily to be read by students, I am pleased 
to dedicate this note to the memory of Nicole De Grande-De Kimpe,
to Chung-Chun Yang, and to Alain Escassut. 
Over the courses of their careers, each of these individuals
has been encouraging and supportive
of students and young mathematicians throughout the world.

\smallskip
I would like to thank Alain Escassut for 
suggesting I cite the work of Lazard and Salmon.
I would also like to thank the
anonymous referee for suggesting some improvements to this manuscript,
and in particular for pointing out that 
a somewhat lengthy ad-hoc proof of one of the implications
of Proposition~\ref{unit}
that I had in an early draft was not needed.

\medskip
Let $\mathbf{F}$ be an algebraically closed field complete with
respect to a non-trivial non-Archimedean absolute value, which
we denote by $|~|.$  Denote by 
$|\mathbf{F}^\times|$ the value group of $\mathbf{F},$
or in other words 
$$
	|\mathbf{F}^\times|=\{|a| : a\in\mathbf{F}^\times=\mathbf{F}\setminus\{0\}\}.
$$
Because $|~|$ is non-trivial and $\mathbf{F}$ is algebraically closed,
$|\mathbf{F}^\times|$ is dense in the positive real numbers.
Let $\mathbf{B}^m(r)$ denote the ``closed'' ball of radius $r$
in $\mathbf{F}^m,$ \textit{i.e.,}
$$
	\mathbf{B}^m(r)=\{(z_1,\dots,z_m)\in\mathbf{F}^m :
	\max |z_i| \le r\}.
$$
\textit{Henceforth, we will only consider $r\in|\mathbf{F}^\times|.$}
Denote by
$\mathcal{A}^m(r)$ the ring of analytic functions
on $\mathbf{B}^m(r),$ or in other words the sub-ring of formal
power series in the multi-variable $z=(z_1,\dots,z_m)$ with coefficients
in $\mathbf{F}$ converging on $\mathbf{B}^m(r),$ \textit{i.e.,}
$$
	\mathcal{A}^m(r) = \left\{\sum_\gamma a_\gamma z^\gamma :
	\lim_{|\gamma|\to\infty}|a_\gamma|r^{|\gamma|}=0\right\}.
$$
Note that we use multi-variable and multi-index notation throughout,
and that for a multi-index $\gamma=(\gamma_1,\dots,\gamma_m),$ we use
$|\gamma|$ to mean
$$
	|\gamma|=\gamma_1+\dots+\gamma_m.
$$
We recall that a multi-index \hbox{$\alpha=(\alpha_1,\dots,\alpha_m)$}
is said to be greater than
a multi-index \hbox{$\beta=(\beta_1,\dots,\beta_m)$}
in the \textbf{graded lexicographical order}
if \hbox{$|\alpha|>|\beta|$} or if \hbox{$|\alpha|=|\beta|$} and $\alpha$ is
greater than $\beta$ in the (ungraded) lexicographical ordering, which means
that for the smallest subscript $i$ such that $\alpha_i\ne\beta_i,$
we have that $\alpha_i>\beta_i.$ Comparing multi-indices or monomials
based on the graded lexicographical order simply means to first compare
the total degree and then to break ties between monomials
of the same total degree by using the lexicographical order.

Denote the quotient field of $\mathcal{A}^m(r),$ \textit{i.e.,}
the field of meromorphic functions on $\mathbf{B}^m(r),$
by $\mathcal{M}^m(r).$
We will also want to consider analytic and meromorphic functions 
that do not depend on the final variable $z_m,$ and 
for convenience, we denote
these by $\mathcal{A}^{m-1}(r)$ and $\mathcal{M}^{m-1}(r).$

Recall that the residue class field $\widetilde{\mathbf{F}}$ is defined by
$$
	\widetilde{\mathbf{F}}=\{a \in \mathbf{F} : |a|\le1\}/
	\{a \in \mathbf{F} : |a| < 1\}.
$$
A property will said to be true for an
$m$-tuple \hbox{$u=(u_1,\dots,u_m)$}
\textbf{over a generic residue class}
if $|u_j|\le 1$ for $1\le j \le m$ and
if the property holds for all such $u$ such that the reduction
\hbox{$\tilde u = (\tilde u_1,\dots,\tilde u_m)$}
lies outside the zero locus in $\widetilde{\mathbf{F}}^m$
of some non-zero polynomial
in $m$ variables with coefficients in $\widetilde{\mathbf{F}};$
note that $\widetilde{\mathbf{F}}$ is algebraically closed.

If
$$
	f(z)=\sum a_\gamma z^\gamma
$$
is an element of $\mathcal{A}^m(r),$ then denote by
$$
	|f|_r=\sup_\gamma |a_\gamma|r^{|\gamma|}.
$$
We begin with the non-Archimedean maximum modulus principle.

\begin{proposition}[Maximum Modulus Principle]\label{maxmod}
Let $f$ be an analytic function in $\mathcal{A}^m(r).$
Then, $|f(z)|\le|f|_r$ for all $z$ in $\mathbf{B}^m(r).$
Moreover, let $c$ be an element of $\mathbf{F}$ with $|c|=r.$
Then for \hbox{$u=(u_1,\dots,u_m)$} over a generic residue class,
$$
	|f(cu_1,\dots,cu_m)|=|f|_r.
$$
\end{proposition}

\begin{proof} See \cite[Prop.~5.1.4/3]{BGR}. I give the argument here
because a solid understanding of $|f|_r$ is fundamental to most of what
I do in this note.  Write
$$
	f(z)=\sum_\gamma a_\gamma z^\gamma.
$$
Then, we immediately have,
$$
	|f(z)| = \left|\sum_\gamma a_\gamma z^\gamma\right|
	\le \sup |a_\gamma||z^\gamma|
	\le \sup |a_\gamma|r^{|\gamma|}=|f|_r.
$$
To see that equality holds for $u$ above a generic residue class,
let $\Gamma$ be the set of multi-indices $\gamma$ such that
\hbox{$|a_\gamma|r^{|\gamma|}=|f|_r.$} Let $\gamma_0$ be a multi-index
in $\Gamma,$ and let $b=a_{\gamma_0}c^{|\gamma_0|}$ so that
$|b|=|f|_r.$
If
$$
	|f(cu_1,\dots,cu_m)| < |f|_r,
$$
then
$$
	\left|\sum_{\gamma\in\Gamma}a_\gamma c^{|\gamma|}u^\gamma\right|
	< |f|_r,
$$
and hence
$$
	\left|\sum_{\gamma\in\Gamma}e_\gamma u^\gamma\right|<1, 
	\qquad\textnormal{where~}e_\gamma=\frac{a_\gamma c^{|\gamma|}}{b}.
$$
Note that $|e_\gamma|\le1$ and that $e_{\gamma_0}=1.$ In terms of residue
classes, the previous inequality precisely means
$$
	\sum_{\gamma\in\Gamma} \tilde e_\gamma {\tilde u}^\gamma=0.
$$
This is a non-trivial polynomial relation, and hence we must have equality
over a generic residue class.
\end{proof}

\begin{corollary}\label{mult}
The real-valued function $|~|_r$ on $\mathcal{A}^m(r)$ is a
non-Archimedean absolute value on $\mathcal{A}^m(r).$
\end{corollary}

\begin{proof}
Let $f$ and $g$ be analytic functions in $\mathcal{A}^m(r).$
That
$$
       |f+g|_r\le\max\{|f|_r,|g|_r\}
$$
follows directly from
the fact that $|~|$ is a non-Archimedean absolute value on
$\mathbf{F}.$ To check multiplicativity, note that by
Proposition~\ref{maxmod}, there exists a point $(a_1,\dots,a_m)$
in $\mathbf{B}^m(r)$ such that
\begin{align*}
	|f|_r &= |f(a_1,\dots,a_m)|,\\
	|g|_r &= |g(a_1,\dots,a_m)|, \textnormal{~and}\\
	|fg|_r &= |f(a_1,\dots,a_m)g(a_1,\dots,a_m)|,
\end{align*}
and so the multiplicativity of $|~|_r$
also follows from the multiplicitivity of $|~|.$
\end{proof}

Note that we may extend $|~|_r$ to a non-Archimedean
absolute value on $\mathcal{M}^m(r)$ by multiplicativity.

\begin{proposition}[{\cite[Th.~31.14]{Escassut}}]\label{unit}
An analytic function of the form
$$
	u(z)=1+\sum_{|\gamma|\ge1}a_\gamma z^\gamma
$$
is a unit in $\mathcal{A}^m(r)$ if and only if
$$
	\sup_{|\gamma|\ge1}|a_\gamma|r^{|\gamma|}<1.
$$
\end{proposition}

\begin{proof}
If $u=1-f,$ where $f$ is in $\mathcal{A}^m(r)$
with $|f|_r<1,$ then
$$
	1+f+f^2+f^3+\dots
$$
converges to a function $v$ such that $uv=1,$ and so $u$ is a unit.
\par
We will postpone the proof of the converse until later.
\end{proof}

Following \cite{BGR}, but working with $\mathcal{A}^m(r)$
instead of just $\mathcal{A}^m(1),$ we say that an analytic function
$$
	f(z)=f(z_1,\dots,z_m)=\sum_\gamma a_\gamma z^\gamma
	= \sum_{j=0}^\infty A_j(z_1,\dots,z_{m-1})z_m^j
$$\goodbreak\noindent
in $\mathcal{A}^m(r)$ thought of as a power series in $z_m$ alone with
coefficients in $\mathcal{A}^{m-1}(r)$ is
\textbf{\hbox{$z_m$-distinguished} of degree $n$}
\begin{itemize}
\item if $A_n(z_1,\dots,z_{m-1})$
is a unit in $\mathcal{A}^{m-1}(r),$ 
\item if \hbox{$|f|_r=|A_n|_rr^n,$} 
\item and if
\hbox{$|A_j|_rr^j < |A_n|_rr^n$} for all $j>n.$
\end{itemize}
The function $f$ is called
simply $z_m$-distinguished if it is $z_m$-distinguished of degree $n$
for some $n\ge0.$ Note that if $f$ is $z_m$-distinguished of degree $0,$
then $f$ is a unit in $\mathcal{A}^m(r)$ by 
Proposition~\ref{unit}.
An element $W$ of
$\mathcal{A}^{m-1}(r)[z_m],$ \textit{i.e.,} a polynomial in the last
variable $z_m$ with coefficients analytic, but not necessarily polynomial,
in the first $m-1$ variables, of degree $n$ in $z_m$
is called a \textbf{Weierstrass polynomial}
if $W$ is monic and if $|W|_r=r^n.$

\begin{proposition}\label{Wnotunit}
A Weierstrass polynomial of positive degree is not a unit.
\end{proposition}

\begin{proof}
Let
$$
    W(z_1,\dots,z_m)=A_0(z_1,\dots,z_{m-1})+\dots+A_{d-1}(z_1,\dots,z_{m-1})z_m^{d-1}+z_m^d
$$
be a Weierstrass polynomial of degree $d>0$ in
$\mathcal{A}^{m-1}(r)[z_m].$ Factor the one-variable monic polynomial
$$
    W(0,\dots,0,z_m)= (z_m-b_1)\cdots(z_m-b_d).
$$
I claim that for some $j$ from $1$ to $d,$ we must have that
$|b_j|\le r.$ For if not, then
$$
   |W(0,\dots,0,z_m)|_r = |z_m-b_1|_r\cdots|z_m-b_d|_r > r^d,
$$
which, by Proposition~\ref{maxmod}, contradicts the hypothesis
that $|W|_r=r^d.$ Hence, there is some $b$ with $|b|\le r$ such that
$W(0,\dots,0,b)=0,$ and hence $W$ is not a unit, as was to be shown.
\end{proof}

I now state the important
\begin{theorem}[{Weierstrass Preparation Theorem \cite[Th.~5.2.2/1]{BGR}}]\label{wprep}
If an analytic function $f$ in $\mathcal{A}^m(r)$
is $z_m$-distinguished of degree $n,$
then there is a unique Weierstrass polynomial 
\hbox{$W\in\mathcal{A}^{m-1}(r)[z_m]$}
of degree $n$ and a unique unit $u$ in $\mathcal{A}^m(r)$ such that
$f=uW.$
\end{theorem}

\begin{proposition}\label{wpfactor}
Let $f_1,f_2 ,W\in\mathcal{A}^{m-1}(r)[z_m]$ such that $W=f_1f_2$ and such that
$W$ is a Weierstrass polynomial. Then, there exist units $u_1$ and $u_2$
in $\mathcal{A}^{m-1}(r)$ such that $f_j/u_j$ are also Weierstrass polynomials.
\end{proposition}

\begin{proof}
This proof is similar to \cite[Lemma~6.4.8]{Krantz}.

Let $d,$ $d_1$ and $d_2$ be the degrees of $W,$ $f_1,$ and $f_2$
respectively thought of as polynomials in $z_m.$ For $j=1,2,$ write
$$
	f_j=A_{j,d_j}(z_1,\dots,z_{m-1})z_m^{d_j}+\dots+
	A_{j,0}(z_1,\dots,z_{m-1}).
$$
Then, because $W$ is monic,
$$
	z_m^d+\dots=\left(A_{1,d_1}z_m^{d_1}+\dots+
	A_{1,0}\right)\cdot
	\left(A_{2,d_2}z_m^{d_1}+\dots+
	A_{2,0}\right),
$$
and hence $A_{1,d_1}$ and $A_{2,d_2}$ are units in
$\mathcal{A}^{m-1}(r),$ 
$$
	|A_{1,d_1}|_r\cdot|A_{2,d_2}|_r=1,
$$
and
$$
	\max_{\begin{array}{c}\scriptstyle 0\le i_1 \le d_1\\ \scriptstyle 
	0\le i_2 \le d_2\end{array}}
	|A_{1,i_1}|_r\cdot|A_{2,i_2}|_r r^{i_1+i_2} \le r^d.
$$
For $j=1,2,$ let $W_j=f_j/A_{j,d_j}.$ Then, $W_1$ and $W_2$
are monic and if $\{j,k\}=\{1,2\},$ we have
$$
	\left|\frac{A_{j,i}}{A_{j,d_j}}\right|_r r^i =
	|A_{j,i}|_r\cdot|A_{k,d_k}|_r \cdot r^i
	\cdot\frac{r^{d_k}}{r^{d_k}} \le \frac{r^d}{r^{d_k}}=r^{d_j},
$$
which precisely means that $W_1$ and $W_2$ are Weierstrass polynomials.
\end{proof}

As the Weierstrass Preparation Theorem only applies to 
$z_m$-distinguished functions, we need to know that every function can
be made to be $z_m$-distinguished after a simple change of variables.
The standard reference \cite[Prop.~5.2.4/2]{BGR} uses a non-linear
coordinate change, but a linear coordinate change will be more useful
for our purposes here. Let \hbox{$u=(u_1,\dots,u_{m-1})$} be an
$m-1$-tuple of elements $u_j$ in $\mathbf{F}$ with $|u_j|\le1.$
We consider the $\mathbf{F}$-algebra automorphism $\sigma_u$
of $\mathcal{A}^m(r)$ defined by
$$
	\sigma_u(z_1,\dots,z_m)=
	(z_1+u_1z_m,\dots,z_j+u_jz_m,\dots,z_{m-1}+u_{m-1}z_m,z_m).
$$
The homomorphism $\sigma_u$ is easily seen to be an automorphism
by observing that its inverse is given by
$$
	\sigma_u^{-1}(z_1,\dots,z_m)=
	(z_1-u_1z_m,\dots,z_j-u_jz_m,\dots,z_{m-1}-u_{m-1}z_m,z_m).
$$
\par
\begin{proposition}\label{changevars}
Let $r,R\in|\mathbf{F^\times}|,$ with $r\le R,$ and let 
$$
	f(z)=\sum_\gamma a_\gamma z^\gamma \in \mathcal{A}^m(R)
$$
be such that $f$ is not identically zero.  Then for 
an $m-1$-tuple $u$ over a generic residue class,
$f\circ\sigma_u$ is $z_m$-distinguished
in $\mathcal{A}^m(r).$
\end{proposition}

\begin{remark*}
I emphasize that because we can choose $u$ over a generic
residue class, 
given any finite collection of functions $f_k$ and given any
finite number of radii $r_\ell\in|\mathbf{F}^\times|$ with
$r_\ell\le R,$ we can find an automorphism $\sigma_u$ so that
the $f_k\circ\sigma_u$ are all simultaneously $z_m$-distinguished
in each of the rings $\mathcal{A}^m(r_\ell).$ In fact, we can do
this simultaneously for all $r\le R,$ but we will not need that.
\end{remark*}

\begin{proof}
Write
$$
	f\circ\sigma_u(z_1,\dots,z_m)=\sum_{j=0}^\infty
	B_j z_m^j.
$$
Each $B_j$ is a power-series with integer coefficients in the
$a_\gamma,$  in the $u_j,$ and in the variables $z_1,\dots,z_{m-1}.$
Those coefficients $a_\gamma$ which appear in $B_j$ are precisely
those with \hbox{$\gamma=(\gamma_1,\dots,\gamma_m),$} where
$|\gamma|\ge j$ and $\gamma_m\le j.$
Let $\mu$ be the largest multi-index in the graded lexicographical order
such that $|a_\mu|r^{|\mu|}=|f|_r.$

Consider $j>|\mu|.$ In this case, all the coefficients
$a_\gamma$ appearing in $B_j$ are such that
\hbox{$|a_\gamma|r^{|\gamma|} < |f|_r.$}
Thus, for $j>|\mu|,$
$$
	|B_j|_rr^j \le
	\sup_{\gamma}|a_\gamma|r^{|{\gamma}|} < |f|_r,
$$
where the $\sup$ is taken over those $\gamma$ 
with $a_\gamma$ appearing in $B_j,$ all of which have graded lexicographical
order greater than $\mu.$

For $j=|\mu|,$ note that any term appearing in $B_j$ that involves
any of the variables $z_1,\dots,z_{m-1}$ will include a coefficient
$a_\gamma$ with $\gamma$ greater that $\mu$ in the 
graded lexicographical ordering,
and thus will have 
$$
	|a_\gamma| r^{|\gamma|} < |a_\mu| r^{|\mu|}.
$$
On the other hand, one of the constant terms appearing in $B_j$ is
$$
	a_\mu u_1^{\mu_1}\cdots u_{m-1}^{\mu_{m-1}}.
$$
Thus, keeping in mind we are considering $j=|\mu|,$
$$
	|B_j|_r r^j = |a_\mu| r^{|\mu|} = |f|_r,
$$
provided none of the other constant terms in $B_j$ reduce the norm of
$B_j,$ and this is true for $u$ over a generic residue class.
Also, because the 
norm of the constant term in $B_j$ dominates all the norms of the variable
terms, $B_j$ is a unit in $\mathcal{A}^{m-1}(r)$ by Proposition~\ref{unit}.
Note that here we only use the implication in Propostion~\ref{unit} that
we have already proven.

For $j<\mu,$ we have
$$
	|B_j|_rr^j \le
	\sup_{\gamma}|a_\gamma|r^{|{\gamma}|} 
	\le |a_\mu|r^{\mu}=|f|_r,
$$
where again the $\sup$ is taken over those $\gamma$ appearing
in $B_j.$  Hence, we conclude that, for  $u$ over a generic
residue class, $|f\circ\sigma_u|_r = |f|_r$
and that $f\circ\sigma_u$ is $z_m$-distinguished in
$\mathcal{A}^m(r).$
\end{proof}

\begin{proof}[Completion of the proof of Propostion~\ref{unit}.]
Recall that we are in the situation where
$$
   u(z)=1+\sum_{|\gamma|\ge1}a_\gamma z^\gamma.
$$
We need to show that if
\begin{equation}\label{supgeone}
     \sup_{|\gamma|\ge1} |a_\gamma|r^{|\gamma|}\ge1,
\end{equation}
then $u$ is not a unit. By Proposition~\ref{changevars},
we may assume that $u$ is $z_m$-distinguished,
and of positive degree by~(\ref{supgeone}).
Theorem~\ref{wprep} then says that we can write $u=vW,$
where $v$ is a unit and $W$ is a Weierstrass polynomial of positive
degree. Propostion~\ref{Wnotunit} then implies that $u$ is not a unit.
\end{proof}

\begin{theorem}[{\cite[Th.~5.2.6/1]{BGR}}]\label{ufd}
The ring $\mathcal{A}^m(r)$ is factorial.
\end{theorem}

\begin{remark*} This was proven by Salmon in \cite{Salmon}.
\end{remark*}

\begin{proposition}\label{relprime}
Let $r<R$ with $r$ and $R$ in $|\mathbf{F}^\times|.$
Let $f_1$ and $f_2$ be analytic functions in
\hbox{$\mathcal{A}^m(R)\subsetneq\mathcal{A}^m(r).$}
If $f_1$ and $f_2$ are relatively prime in the ring
$\mathcal{A}^m(R),$ they remain relatively prime when considered
as elements of the bigger ring $\mathcal{A}^m(r),$ in other words
when they are restricted to $\mathbf{B}^m(r).$
\end{proposition}

\begin{proof}
This is a standard argument.  See, for instance \cite[Prop.~6.4.11]{Krantz},
where the analogous result for germs of analytic functions on a domain
in $\mathbf{C}^m$ is proven.

We can multiply by units and make changes of variables without changing
the question as to whether two functions are relatively prime.
Hence, using Proposition~\ref{changevars} (and the remark
following it) and Theorem~\ref{wprep},
we may assume without loss of generality that $f_1$ and $f_2$
are Weierstrass polynomials relatively prime in $\mathcal{A}^m(R),$
and that they are $z_m$-distinguished in $\mathcal{A}^m(r).$

We now claim that $f_1$ and $f_2$ are relatively prime in 
$\mathcal{A}^{m-1}(R)[z_m].$  The novice reader should think about
why this is not an entirely trivial statement because although
$\mathcal{A}^{m-1}(R)[z_m]$ is a smaller ring than $\mathcal{A}^m(R),$
there are also fewer units. Indeed, suppose there is a 
non-trivial common factor $h$
and functions $g_1$ and $g_2$ in $\mathcal{A}^{m-1}(R)[z_m]$ such that
$f_j=hg_j.$  Then by Proposition~\ref{wpfactor}, $h,$ $g_1$ and $g_2$
are also Weierstrass polynomials, up to units. By assumption $h$ is
not a unit in $\mathcal{A}^{m-1}(R)[z_m],$ and hence is not of degree 0.
Because, up to a unit,
$h$ is a Weierstrass polynomial, this means $h$ is also not
a unit in $\mathcal{A}^m(R)$ contradicting our original assumption that
$f_1$ and $f_2$ are relatively prime in $\mathcal{A}^m(R).$

Now, by Gauss's Lemma, $f_1$ and $f_2$ are relatively prime
in $\mathcal{M}^{m-1}(R)[z_m].$  This is important, because
$\mathcal{M}^{m-1}(R)[z_m],$ being a one-variable polynomial ring
over a field, is a principal ideal domain.  Hence, there exist
$G_1$ and $G_2$ in $\mathcal{M}^{m-1}(R)[z_m]$ such that
$$
	1=G_1 f_1 + G_2 f_2.
$$
Clearing denominators, we find functions $h,$ $g_1$ and $g_2$
in $\mathcal{A}^{m-1}(R)$ such that
$$
	h=g_1f_1+g_2f_2.
$$

Finally, suppose that $f_1$ and $f_2$ are not relatively prime in
$\mathcal{A}^m(r).$ Then, there is a non-trivial common factor $\bar f$
of $f_1$ and $f_2$ in $\mathcal{A}^m(r).$  Since $f_1$ 
is $z_m$-distinguished in $\mathcal{A}^m(r),$ we know by
Theorem~\ref{wprep} that it is a unit times a Weierstrass polynomial
in $\mathcal{A}^{m-1}(r)[z_m].$  Because $\bar f$ is a factor of $f_1,$
we can then use Proposition~\ref{wpfactor} to conclude that
$\bar f$ is a unit times a Weierstrass polynomial.  Thus, we might as
well assume $\bar f$ is a Weierstrass polynomial.  But $\bar f,$
being a common factor of $f_1$ and $f_2,$ divides $h,$ which does not
depend on $z_m.$  Hence $\bar f$ has degree zero as a Weierstrass polynomial,
and is therefore a unit in $\mathcal{A}^{m-1}(r).$
\end{proof}

\begin{corollary}\label{gcd}
Let $r<R$ with $r$ and $R$ in $|\mathbf{F}^\times|.$
Let $f_1,\dots,f_k$ be analytic functions in
\hbox{$\mathcal{A}^m(R)\subsetneq\mathcal{A}^m(r).$}
If $G$ is a greatest common divisor of the $f_j$
in $\mathcal{A}^m(R)$ and if $g$ is a greatest common divisor
of the $f_j$ in $\mathcal{A}^m(r).$  Then considering $g$ and $G$
as elements of $\mathcal{A}^m(r),$ they differ by a unit in
$\mathcal{A}^m(r).$
\end{corollary}

\begin{proof}
By induction, we need only consider the case $k=2.$
Clearly $G$ divides $g.$
By assumption $f_1/G$ and $f_2/G$ are relatively prime
in $\mathcal{A}^{m}(R).$ By the proposition they remain relatively
prime in $\mathcal{A}^m(r).$ Hence $g$ divides $G.$
\end{proof}

Let $r<R$ with both $r$ and $R$ in $|\mathbf{F}^\times|.$
Let $P$ be an irreducible element in $\mathcal{A}^m(R).$
If we restrict $P$ to an element of $\mathcal{A}^m(r),$ one of
three things can happen: $P$ may remain irreducible, $P$ may become
a unit, or $P$ may become reducible. As an example of the second
case, consider $P(z)=1-z$ in one variable. If $r<1<R,$
then $P$ is irreducible in $\mathcal{A}^1(R)$ but a unit in
$\mathcal{A}^1(r).$ The third possibility is strictly a several variable
phenomenon. For example, consider $P(z_1,z_2)=z_2^2-z_1^2(1-z_1).$
Then, $P$ is irreducible for $R$ large. However, for $r<1,$ we can find
an analytic branch of $\sqrt{1-z_1},$ and hence $P$ factors as
$$
	P(z_1,z_2)=(z_2-z_1\sqrt{1-z_2})(z_2+z_1\sqrt{1-z_1}).
$$
However, we do have the following useful corollary.

\begin{corollary}\label{irreducible}
Let $r<R$ be in $|\mathbf{F}^\times|.$ Let $f$ be an element
of $\mathcal{A}^m(R).$
Let $q$ be an irreducible factor of $f$ 
in $\mathcal{A}^m(r).$ Then, up to multiplication by unit in
$\mathcal{A}^m(R),$
there exists a unique irreducible factor
$Q$ of $f$ in $\mathcal{A}^m(R)$ such that $q$ divides $Q$ in
$\mathcal{A}^m(r).$ Moreover, $q$ divides $Q$ with exact multiplicity $1,$
and $Q$ divides $f$ in $\mathcal{A}^m(R)$ with the same exact multiplicity
with which $q$ divides $f$ in $\mathcal{A}^m(r).$
\end{corollary}

\begin{proof} 
Using Theorem~\ref{ufd}, write $f=p_1^{d_1}\cdots p_s^{d_s}$
in $\mathcal{A}^m(r)$ and $f=P_1^{e_1}\cdots P_t^{e_t}$ in
$\mathcal{A}^m(R),$ with the $p_i$ and the $P_j$ distinct
irreducible elements. Without loss of generality, assume $q=p_1.$

I will first show that $q$ divides at most one $P_j$ 
in $\mathcal{A}^m(r).$ Since $P_j$ and $P_k$ are irreducible in
$\mathcal{A}^m(R),$ they are relatively prime in $\mathcal{A}^m(R).$
If $q$ were to divide $P_j$ and $P_k$ with
$k\ne j,$ then $P_j$ and $P_k$ would not be relatively prime in
$\mathcal{A}^m(r),$  which would contradict Proposition~\ref{relprime}.

Since $q$ is irreducible in $\mathcal{A}^m(r)$ and divides
$f=P_1^{e_1}\cdots P_t^{e_t},$ it must clearly divide one of the $P_j,$
which, without loss of generality, we will assume is $P_1.$

It remains to check that $q$ divides $P_1$ with multiplicity $1,$ as this
will then imply that $d_1=e_1.$ Because $P_1$ is not a unit and irreducible,
there exists a $j$ with $1\le j \le m$ such that
$\partial P_1/\partial z_j\not\equiv0.$ In characteristic zero, this
follows from the fact that any non-constant analytic function has
at least one partial derivative which does not vanish identically.
In positive characteristic $p,$ if all the partial derivatives vanish
identically, then the analytic function is a pure $p$-th power, and hence
not irreducible. Since $P_1$ is irreducible, it must be relatively
prime to $\partial P_1/\partial z_j.$ Again, by Propositon~\ref{relprime},
$P_1$ and $\partial P_1/\partial z_j$ remain relatively prime in
$\mathcal{A}^m(r).$ Thus, no irreducible element in $\mathcal{A}^m(r)$
can divide $P_1$ with multiplicity greater than one.
\end{proof}

I now present an argument of L\"utkebohmert \cite{Lu}.

\begin{lemma}{\label{lu}}
For $i=1,2,3,\dots,$ let $r_i$ be an increasing sequence of elements in
$|\mathbf{F}^\times|$ such that $r_i\to\infty.$
Suppose that for each $i,$
we are given analytic functions
$g_i$ in $\mathcal{A}^m(r_i)$ and for each $i<j,$ we are given
units $u_{i,j}$ in $\mathcal{A}^m(r_i)$ such that in $\mathcal{A}^m(r_i)$
we have
$$
	g_i = u_{i,j}g_j.
$$
Then, there exists an entire function $G$ on $\mathbf{F}^m$ and units
$v_i$ in $\mathcal{A}^m(r_i)$ such that $g_i=Gv_i$ in $\mathcal{A}^m(r_i).$
\end{lemma}

\begin{remark*} Since $g_i=Gv_i,$ we see that for $j\ge i,$
\begin{equation}\label{eqeqn}
	g_iv_i^{-1}=G=g_jv_j^{-1} \qquad\textnormal{in~}
	\mathcal{A}^m(r_i),
\end{equation}
\end{remark*}

\begin{proof}
If one of the $g_i$ is identically zero, then they all are,
and we can clearly take $g\equiv0$ and $v_i\equiv1.$ Thus, 
we may assume that there exists
a point $z_0$ in $\mathbf{B}^m(r_1)$
such that $g_1(z_0)\ne0,$ and hence $g_j(z_0)\ne0$ for
all $j$ since $g_1$ and $g_j$ differ by a unit. Without loss of generality,
we may adjust the $g_i$ by multiplicative constants so that
$g_i(z_0)=1$ for all $i.$ This of course implies that
$u_{i,j}(z_0)=1$ for all $i<j$ too. Now, expand $u_{i,i+1}$ as
a power series about $z_0$ to get
$$
	u_{i,i+1}(z)=1+\sum_{|\gamma|\ge1}a_\gamma(z-z_0)^\gamma,
	\qquad\textnormal{with~}|a_\gamma|r_i^{|\gamma|} < 1
	\textnormal{~for all~}\gamma,
$$
by Proposition~\ref{unit}. Hence, for $j>i,$
$$
	|u_{j,j+1}-1|_{r_i} < \frac{r_i}{r_j}.
$$
Fixing $i$ and letting $j\to\infty,$ we have $r_i/r_j\to0,$
and so we can use an infinite product to define a unit $v_i$
in $\mathcal{A}^m(r_i)$ by
$$
	v_i=\prod_{k=i}^\infty u_{k,k+1}.
$$
For $j>i,$ note that
\begin{align*}
	g_jv_i&=g_j\prod_{k=i}^\infty u_{k,k+1}\\
	&=g_j\left(\prod_{k=1}^{j-1}u_{k,k+1}\right)
	\left(\prod_{k=j}^\infty u_{k,k+1}\right) = g_iv_j.
\end{align*}
Therefore, for all $i\le j,$ we have
$$
	g_iv_i^{-1}=g_jv_j^{-1} \qquad\textnormal{in~}
	\mathcal{A}^m(r_i),
$$
which is precisely~(\ref{eqeqn}) and which also means that the
$g_iv_i^{-1}$ converge to an entire function $G$ such that
$G=g_iv_i^{-1}$ in $\mathcal{A}^m(r_i).$
\end{proof}

Now using L\"utkebohmert's argument as
presented in the appendix to \cite{CY}, we get the key result
of this note and what was needed in \cite{CY}.

\begin{theorem}\label{entire}
Greatest common divisors exist in the ring
of entire functions on $\mathbf{F}^m.$ Moreover, if $G$ is the
greatest common divisor of the entire functions $f_1,\dots,f_k$
in the ring of entire functions, then $G$ is also the greatest common
divisor of $f_1,\dots,f_k$ in the ring $\mathcal{A}^m(r)$ for all
$r\in|\mathbf{F}^\times|.$
\end{theorem}

\begin{proof}
It suffices to prove the theorem when $k=2.$

Let $f_1$ and $f_2$ be two entire functions
on $\mathbf{F}^m.$ If $f_1$ is identically zero,
then clearly $f_2$ is a greatest common divisor of $f_1$ and $f_2.$
Thus, we now assume $f_1$ is not identically zero.

Let $r_i\in|\mathbf{F}|$ for $i=1,2,3,\dots$ be an increasing sequence with
$r_i\to\infty.$ Of course $f_1$ and $f_2$ are also elements
of each of the factorial rings $\mathcal{A}^m(r_i).$  Hence,
for each $i,$
there exist analytic functions $g_i$ in $\mathcal{A}^m(r_i)$
such that $g_i$ is a greatest common divisor of $f_1$ and $f_2$
in $\mathcal{A}^m(r_i).$ For any $i<j,$ by Corollary~\ref{gcd},
there exists a unit $u_{i,j}$ in $\mathcal{A}^m(r_i)$ such that
$$
	g_i=u_{i,j}g_j
$$
in $\mathcal{A}^m(r_i).$ Now, let $v_i$ and $G$ be as
in Lemma~\ref{lu}. Since $g_i=Gv_i,$ we see that $G$ and $g_i,$
differing by a unit, are both greatest common divisors of
$f_1$ and $f_2$ in $\mathcal{A}^m(r_i).$ By Corollary~\ref{gcd},
this also implies that $G$ is the greatest common divisor of $f_1$ and
$f_2$ in $\mathcal{A}^m(r)$ for all $r$ in $|\mathbf{F}^\times|$ with
$r\le r_i.$

It remains to show that $G$ is a greatest common divisor for
$f_1$ and $f_2$ in the ring of entire functions on $\mathbf{F}^m.$
We first check that $G$ divides $f_1.$ Since $g_i$ is a factor of
$f_1$ in $\mathcal{A}^m(r_i),$ there exist analytic functions
$h_i$ in $\mathcal{A}^m(r_i)$ such that $f_1=g_ih_i.$
By (\ref{eqeqn}), $h_iv_i$ converge to an entire function $H$
such that $f_1=GH,$ and hence $G$ is a factor of $f_1.$
Similarly, $G$ is a factor of $f_2,$ and so $G$ is a common factor.

Now let $g$ be any other entire common factor of $f_1$ and $f_2.$
Because $g_i$ is a greatest common factor in $\mathcal{A}^m(r_i),$
there exist analytic functions $\omega_i$ such that
\hbox{$g_i=g\omega_i$} in
$\mathcal{A}^m(r_i).$ Thus, $g_iv_i^{-1}=g\omega_iv_i^{-1}$ in
$\mathcal{A}^m(r_i).$ Because $\mathcal{A}^m(r_i)$
is an integral domain, equation~(\ref{eqeqn}) implies that
if $i\le j$, then $\omega_iv_i^{-1}=\omega_jv_j^{-1}$
in $\mathcal{A}^m(r_i),$
and so $\omega_iv_i^{-1}$ converges to an entire function
$\Omega$ on
$\mathbf{F}^m$ such that $G=g\Omega.$
\end{proof}

A ring is factorial if each element in the ring can be uniquely
written, up to a permuation and multiplication by units,
as a finite product of irreducible elements.
Although the ring of entire functions on $\mathbf{F}^m$
is plainly not factorial, I will conclude this note by showing
that it is almost as good as factorial. Namely, any entire function
can be written as a (possibly infinite) product of irreducible entire
functions, and the irreducible factors and multiplicities 
in the product are unique, up to permutation and multiplication by units.

\begin{theorem}\label{almostufd}
Let $f$ be a non-zero entire function on $\mathbf{F}^m.$
Then, there exists a countable index set $I,$ for each
$i$ in $I,$ there exist irreducible elements $P_i$ in the ring of entire
functions on $\mathbf{F}^m,$ and for each $i$ in $I,$ there exist
natural numbers $e_i$ such that 
 such that if $i\ne j$, then
$P_i$ and $P_j$ are relatively prime, and such that
$$
	f = \prod_{i\in I} P_i^{e_i}.
$$
Moreover, if $J$ is a countable index set, if for each $j$ in $J,$
there are irreducible entire functions $Q_j,$ and if for each $j$ in $J,$
there are natural numbers $d_j$ such that
$$
	f=\prod_{j\in J}Q_j^{d_j}
$$
and such that for $i\ne j$ in $J,$ we have $Q_i$ and $Q_j$ relatively
prime, then there is a bijection $\sigma:I\rightarrow J$ such that
$P_i = Q_{\sigma(i)}$ and $e_i=d_{\sigma(i)}.$
\end{theorem}

\begin{remark*} I recall here that the meaning of the infinite products
in Theorem~\ref{almostufd} is that the finite partial products converge
to $f$ in $\mathcal{A}^m(r)$ for all $r\in|\mathbf{F}^\times|.$
\end{remark*}

I will begin with a proposition describing how to find the irreducible
factors.

\begin{proposition}\label{existence}
Let $f$ be a non-zero entire function on $\mathbf{F}^m.$
Let $r_i$ be an increasing sequence of elements of $|\mathbf{F}^\times|$
such that $r_i\to\infty.$ Let $p_{i_0}$ be an irreducible factor of
$f$ in $\mathcal{A}^m(r_{i_0}).$ Then, up to multiplication by a unit,
there exists a unique
irreducible entire function $P$ such that $p_{i_0}$ divides $P$ and
such that $P$ divides $f.$
\end{proposition}

\begin{proof}
I will begin by proving existence. By Corollary~\ref{irreducible},
for each $i\ge i_0,$ there exists a unique irreducible factor
$p_i$ of $f$ in $\mathcal{A}^m(r_i)$ such that $p_{i_0}$ divides $p_i$
in $\mathcal{A}^m(r_{i_0}).$ I now claim that if $j>i,$ then
$p_i$ divides $p_j$ in $\mathcal{A}^m(r_i).$ Indeed, using
Corollary~\ref{irreducible} again, there is a unique irreducible
factor $q_j$ of $f$ in $\mathcal{A}^{m}(r_j)$ such that
$p_i$ divides $q_j$ in $\mathcal{A}^m(r_i).$ But then
$p_{i_0}$ divides $q_j$ in $\mathcal{A}^m(r_{i_0}),$ and so
by uniqueness, $p_j=q_j.$ Now, for each $i$ and for each $j\ge i,$
the function $p_j$ is a factor of $f$ in $\mathcal{A}^m(r_i),$
and so a finite product of finitely many irreducible factors
with bounded multiplicity. Hence, for each $i,$ there exists $J_i$
such that for all $j,k\ge J_i$ we have that $p_j$ and $p_k$ differ
by a unit in $\mathcal{A}^m(r_i).$ For each $i,$ let $f_i$
be $p_j$ restricted to $\mathcal{A}^m(r_i)$ for some $j\ge J_i.$
Now, for each $j\ge i \ge i_0,$ we have units $u_{i,j}$ in
$\mathcal{A}^m(r_i)$ such that $f_i = u_{i,j}f_j.$
By Lemma~\ref{lu}, there exists an entire function $P$ and
units $v_i$ in $\mathcal{A}^m(r_i)$ such that $f_jv_j^{-1}=P.$
I claim that $P$ is an irreducible entire function
which divides $f$ and such that $p_{i_0}$ divides $P$ in
$\mathcal{A}^m(r_{i_0}).$

To see that $P$ divides $f$ note that each $f_i$ divides
$f$ in $\mathcal{A}^m(r_i).$ That means there exist functions
$h_i$ in $\mathcal{A}^m(r_i)$ such that $f_ih_i=f.$
But then for $j\ge i,$ we have $Pv_ih_i=f_ih_i=f=f_jh_j=Pv_jh_j,$
which implies $h_iv_i=h_jv_j.$ Hence,
$h_iv_i$ converges to an entire function $H$ such
that $PH=f$ since $PH=Ph_iv_i=f$ in $\mathcal{A}^m(r_i)$ for all $i.$

Since $P$ restricted to $\mathcal{A}^m(r_{i_0})$ is
$f_{i_0}v_{i_0}^{-1},$ clearly $p_{i_0}$ divides $P$ in
$\mathcal{A}^m(r_{i_0}).$

To see that $P$ is irreducible, suppose that there exist entire
functions $g$ and $h$ such that $P=gh.$ Since $p_{i_0}$ divides
$P$ in $\mathcal{A}^m(r_{i_0}),$ we must have that $p_{i_0}$
divides $g$ or $h,$ so assume without loss of generality that
it divides $g.$ But this implies that $p_i$ divides $g$ in
$\mathcal{A}^m(r_i)$ for all $i\ge i_0,$ and hence $f_i$ divides
$g$ for all $i\ge i_0.$ Thus, $P$ divides $g$ in
$\mathcal{A}^m(r_i)$ for all $i\ge i_0.$ In other words, there exist
$g_i$ in $\mathcal{A}^m(r_i)$ such that $g=Pg_i$
in $\mathcal{A}^m(r_i).$ But, $P=gh=Pg_ih,$ and so
$g_ih=1$ for all $i\ge i_0.$
It then follows that $h$ is a unit in the ring of entire
functions, and so $P$ must be irreducible.

Finally, it remains to check uniqueness. Let $P$ be as constructed
above and suppose there is another irreducible entire function $Q$
such that $p_{i_0}$ divides $Q$ in $\mathcal{A}^m(r_{i_0})$ and such
that $Q$ divides $f$ in the ring of entire functions. As above,
since $p_{i_0}$ divides $Q,$ we have that $p_i$ divides $Q$ for all
$i\ge i_0.$ Hence $f_i$ divides $Q$ for all $i\ge i_0.$ Hence
$P$ divides $Q,$ in which case $P$ and $Q,$ both being irreducible,
differ by a unit, as was to be shown.
\end{proof}

\begin{proof}[Proof of Theorem~\ref{almostufd}.]
For $k=1,2,\dots,$ let
$r_k$ be an increasing sequence of elements in $|\mathbf{F}^\times|$
such that $r_k\to\infty.$ Since $f$ is not identically zero,
let $z_0$ be an element of $\mathbf{B}^m(r_1)$ such that
$f(z_0)\ne0.$ Without loss of generalizty, assume that $f(z_0)=1.$
We proceed to inductively construct
a countable set $\mathcal{P}$ of ordered pairs $(P,e),$
where $P$ is an irreducible entire factor of $f$ and $e$ is a 
natural number. Start by setting $\mathcal{P}=\emptyset.$
Now we add to $\mathcal{P}$ as follows. Let $k$ be the smallest
natural number such that there is an irreducible factor $p$ of $f$
in $\mathcal{A}^m(r_k)$ which does not divide any of the
$P\in\mathcal{P}.$ By Proposition~\ref{existence}, there exists
a unique irreducible entire function $P$ such that $p$ divides $P$
in $\mathcal{A}^m(r_k)$ and such that $P$ divides $f.$
Since $f(z_0)\ne0,$ we also have $P(z_0)\ne0,$ and so we may assume,
without loss of generalizty, that $P(z_0)=1.$ Now if $e$ is the multiplicity
with which $p$ divides $f$ in $\mathcal{A}^m(r_k),$ for a reason similar
to the analagous statement in Corollary~\ref{irreducible},
$P$ divides $f$ with exact multiplicity $e$ in the ring of
entire functions on $\mathbf{F}^m.$ Thus, add the ordered pair
$(P,e)$ to the set $\mathcal{P},$ and repeat the process.
As we have only countably many $r_k$
and only finitely many irreducible factors of $f$ in each
$\mathcal{A}^m(r_k),$ this process will terminate with a countable
set $\mathcal{P}.$

I claim that, up to a unit,
$$
	f = \prod_{(P,e)\in\mathcal{P}} P^e.
$$
Index the elements $(P_i,e_i)$ of $\mathcal{P}$ by a countable index set $I.$
Since any finite product
$$
	\prod_{i=1}^s P_i^{e_i}
$$
divides $f,$ we have entire functions $g_s$ such that
$$
	f = g_s \prod_{i=1}^s P_i^{e_i}.
$$
Also, for each $k,$ there exists $S_k$ such that for all $s\ge S_k,$ we
have that $g_s$ is a unit in $\mathcal{A}^m(r_k).$ Since $g_s(z_0)=1,$
we conclude, by Propostion~\ref{unit},
that for $k\ge j$ and for $s\ge S_k$ that
$$
	|1-g_s|_{r_j} < \frac{r_j}{r_k},
$$
which will tend to zero as $k$ tends to infinity. Since only finitely
many of the $P_i$ are not units in $\mathcal{A}^m(r_j),$ 
there exists $s_0$ such that for $s> s_0,$ we
have, again by Proposition~\ref{unit}, that
$|P_s|_{r_j}=1.$ Hence, we find that
for $k\ge j$ and $s\ge S_k,$ that
$$
	\left|f-\prod_{i=1}^sP_i^{e_i}\right|_{r_j}
	= \left|\prod_{i=1}^{s_0}P_i^{e_i}\right|_{r_j}\cdot
	\left|\prod_{i=s_0+1}^sP_i^{e_i}\right|_{r_j}\cdot
	|g_s-1|_{r_j}
	< \left|\prod_{i=1}^{s_0}P_i^{e_i}\right|_{r_j}\frac{r_j}{r_k}
	\to 0,
$$
and hence the product converges to $f,$ as was to be shown.

To show uniqueness, it suffices to show that if $Q$ is an irreducible
entire function dividing $f,$ then $Q$ is, up to multiplication by
a unit, equal to one of the $P_i$
constructed above. So, suppose $Q$ is an irreducible entire function
dividing $f.$ Let $r_k$ be large enough so that $Q$ is not a unit in
$\mathcal{A}^m(r_k).$ Then, there is an irreducible factor $p$ of $f$
in $\mathcal{A}^m(r_k)$ which divides $Q.$ By construction, there is
a unique $P_i$ such that $p$ divides $P_i.$ Also, by construction
(and by Proposition~\ref{existence}),
$P_i$ divides $Q$ in the ring of entire functions 
since $p$ divides $Q.$ But since $P_i$ and $Q$ are both irreducible
entire functions, we then have that $P_i$ and $Q$ differ by a unit,
as was to be shown.
\end{proof}


\begin{thebibliography}{[M-F]}
\bibitem[BGR]{BGR}  S.~Bosch, U.~G\"untzer and R.~Remmert,
{\it Non-Archi\-me\-de\-an Analysis,} Springer-Verlag, 1984;
MR0746961.
\bibitem[CY]{CY} W.~Cherry and Z.~Ye, Non-Archimedean Nevanlinna
theory in several variables and the non-Archimedean Nevanlinna
inverse problem, \textit{Trans.\ Amer.\ Math.\ Soc.}
\textbf{349} (1997), 5043--5071; MR1407485.
\bibitem[E]{Escassut} A.~Escassut, \textit{Ultrametric Banach Algebras,}
World Scientific, 2003; MR1978961.
\bibitem[Kr]{Krantz} S.~Krantz, \textit{Function Theory of Several
Complex Variables,} Second Edition, Wadsworth \& Brooks/Cole,
1992; MR1162310
\bibitem[La]{Lazard} M.~Lazard, Les z\'eros des fonction analytiques
d'une variable sur un corps valu\'e complet,
\textit{Inst.\ Hautes \'Etudes Sci.\ Publ.\ Math.} \textbf{14} (1962),
47--75; MR0152519.
\bibitem[L\"u]{Lu} W.~L\"utkebohmert, Letter to R.~Remmert, 1995.
\bibitem[S]{Salmon} P.~Salmon, Sur les s\'eries formelles
restreintes, \textit{C.\ R.\ Acad.\ Sci.\ Paris} \textbf{255} (1962),
227--228; MR0144924.
\end{thebibliography}
\end{document}